\documentclass[twoside]{amsart}

\usepackage{amsmath,amsfonts,amsthm,mathrsfs}
\usepackage{amssymb}

\usepackage[unicode,bookmarks]{hyperref} 
\usepackage[usenames,dvipsnames]{xcolor}
\hypersetup{colorlinks=false}

\usepackage{expl3}

\ExplSyntaxOn
\prop_new:N \g_cite_map_prop
\tl_new:N \l_citekey_result_tl

\cs_new:Npn \mapcitekey #1#2 {
  \clist_map_inline:nn {#2}
       {  \prop_gput:Nnn  \g_cite_map_prop  {##1} {#1}   }
}

\cs_new:Npn \getcitekey #1 {
   \prop_get:NoN \g_cite_map_prop{#1}  \l_citekey_result_tl
   \quark_if_no_value:NF \l_citekey_result_tl
       {  \tl_set_eq:NN #1  \l_citekey_result_tl  }
}

\cs_new:Npn \showcitekeymaps {\prop_show:N  \g_cite_map_prop }
\ExplSyntaxOff

\usepackage{etoolbox}
\makeatletter
\patchcmd{\@citex}{\if@filesw}{\getcitekey\@citeb \if@filesw}%
    {\typeout{*** SUCCESS ***}}{\typeout{*** FAIL ***}}
\patchcmd{\nocite}{\if@filesw}{\getcitekey\@citeb \if@filesw}%
    {\typeout{*** SUCCESS ***}}{\typeout{*** FAIL ***}}
\makeatother

\usepackage{enumitem}

\newenvironment{aenumerate}{%
	\begin{enumerate}[label=(\alph{*}), ref=(\alph{*})]
}{%
	\end{enumerate}%
}

\usepackage{chngcntr}

\usepackage{tikz}
\usepackage{tikz-cd}
\tikzset{commutative diagrams/arrow style=Latin Modern}

\newcommand{\derR}{\mathbf{R}}
\newcommand{\derL}{\mathbf{L}}

\newcommand{\decal}[1]{\lbrack #1 \rbrack}

\newcommand{\shH}{\mathcal{H}}

\newcommand{\tensor}{\otimes}

\newcommand{\shHom}{\mathcal{H}\hspace{-1pt}\mathit{om}}

\newcommand{\NN}{\mathbb{N}}

\newcommand{\menge}[2]{\bigl\{ \thinspace #1 \thinspace\thinspace \big\vert%
\thinspace\thinspace #2 \thinspace \bigr\}}

\DeclareMathOperator{\rk}{rk}

\DeclareMathOperator{\Spec}{Spec}

\DeclareMathOperator{\id}{id}

\DeclareMathOperator{\Supp}{Supp}
\DeclareMathOperator{\codim}{codim}

\DeclareMathOperator{\Hom}{Hom}

\DeclareMathOperator{\Pic}{Pic}

\newcommand{\define}[1]{\emph{#1}}

\newcommand{\shf}[1]{\mathscr{#1}}
\newcommand{\OX}{\shf{O}_X}

\newcommand{\restr}[1]{\big\vert_{#1}}

\newcommand{\argbl}{-}

\def\overbar#1#2#3{{%
	\setbox0=\hbox{\displaystyle{#1}}%
	\dimen0=\wd0
	\advance\dimen0 by -#2 
	\vbox {\nointerlineskip \moveright #3 \vbox{\hrule height 0.3pt width \dimen0}%
		\nointerlineskip \vskip 1.5pt \box0}%
}}

\newcommand{\into}{\hookrightarrow}

\newcommand{\fu}{f^{\ast}}
\newcommand{\fl}{f_{\ast}}

\newcommand{\iu}{i^{\ast}}

\newcommand{\pu}{p^{\ast}}
\newcommand{\pl}{p_{\ast}}

\newcommand{\tl}{t_{\ast}}

\newcommand{\shF}{\shf{F}}

\newcommand{\shO}{\shf{O}}

\makeatletter
\let\@@seccntformat\@seccntformat
\renewcommand*{\@seccntformat}[1]{%
  \expandafter\ifx\csname @seccntformat@#1\endcsname\relax
    \expandafter\@@seccntformat
  \else
    \expandafter
      \csname @seccntformat@#1\expandafter\endcsname
  \fi
    {#1}%
}
\newcommand*{\@seccntformat@subsection}[1]{%
  \textbf{\csname the#1\endcsname.}
}
\makeatother

\makeatletter
\let\@paragraph\paragraph
\renewcommand*{\paragraph}[1]{%
	\vspace{0.3\baselineskip}%
	\@paragraph{\textit{#1}}%
}
\makeatother

\counterwithin{equation}{subsection}
\counterwithout{subsection}{section}
\counterwithin{figure}{subsection}

\newtheorem{theorem}[equation]{Theorem}
\newtheorem*{theorem*}{Theorem}
\newtheorem{lemma}[equation]{Lemma}
\newtheorem*{lemma*}{Lemma}

\newtheorem*{corollary*}{Corollary}
\newtheorem{proposition}[equation]{Proposition}
\newtheorem*{proposition*}{Proposition}

\newtheorem*{conjecture*}{Conjecture}

\theoremstyle{definition}
\newtheorem{definition}[equation]{Definition}
\newtheorem*{definition*}{Definition}
\theoremstyle{remark}

\newtheorem*{example*}{Example}
\newtheorem*{problem*}{Problem}

\theoremstyle{plain}

\newcommand{\theoremref}[1]{\hyperref[#1]{Theorem~\ref*{#1}}}
\newcommand{\lemmaref}[1]{\hyperref[#1]{Lemma~\ref*{#1}}}
\newcommand{\definitionref}[1]{\hyperref[#1]{Definition~\ref*{#1}}}
\newcommand{\propositionref}[1]{\hyperref[#1]{Proposition~\ref*{#1}}}
\newcommand{\conjectureref}[1]{\hyperref[#1]{Conjecture~\ref*{#1}}}
\newcommand{\corollaryref}[1]{\hyperref[#1]{Corollary~\ref*{#1}}}
\newcommand{\exampleref}[1]{\hyperref[#1]{Example~\ref*{#1}}}
\newcommand{\exerciseref}[1]{\hyperref[#1]{Exercise~\ref*{#1}}}

\makeatletter
\let\old@caption\caption
\renewcommand*{\caption}[1]{%
	\setcounter{figure}{\value{equation}}%
	\stepcounter{equation}%
	\old@caption{#1}\relax%
}
\makeatother

\newcounter{intro}

\newtheorem{intro-conjecture}[intro]{Conjecture}
\newtheorem{intro-corollary}[intro]{Corollary}
\newtheorem{intro-theorem}[intro]{Theorem}

\newcommand{\OA}{\mathscr{O}_A}

\newcommand{\Ah}{\hat{A}}

\newcommand{\newpar}{\subsection{}}
\newcommand{\parref}[1]{\hyperref[#1]{\S\ref*{#1}}}
\newcommand{\chapref}[1]{\hyperref[#1]{Chapter~\ref*{#1}}}

\makeatletter
\newcommand*\if@single[3]{%
  \setbox0\hbox{${\mathaccent"0362{#1}}^H$}%
  \setbox2\hbox{${\mathaccent"0362{\kern0pt#1}}^H$}%
  \ifdim\ht0=\ht2 #3\else #2\fi
  }
\newcommand*\rel@kern[1]{\kern#1\dimexpr\macc@kerna}
\newcommand*\widebar[1]{\@ifnextchar^{{\wide@bar{#1}{0}}}{\wide@bar{#1}{1}}}
\newcommand*\wide@bar[2]{\if@single{#1}{\wide@bar@{#1}{#2}{1}}{\wide@bar@{#1}{#2}{2}}}
\newcommand*\wide@bar@[3]{%
  \begingroup
  \def\mathaccent##1##2{%
    \if#32 \let\macc@nucleus\first@char \fi
    \setbox\z@\hbox{$\macc@style{\macc@nucleus}_{}$}%
    \setbox\tw@\hbox{$\macc@style{\macc@nucleus}{}_{}$}%
    \dimen@\wd\tw@
    \advance\dimen@-\wd\z@
    \divide\dimen@ 3
    \@tempdima\wd\tw@
    \advance\@tempdima-\scriptspace
    \divide\@tempdima 10
    \advance\dimen@-\@tempdima
    \ifdim\dimen@>\z@ \dimen@0pt\fi
    \rel@kern{0.6}\kern-\dimen@
    \if#31
      \overline{\rel@kern{-0.6}\kern\dimen@\macc@nucleus\rel@kern{0.4}\kern\dimen@}%
      \advance\dimen@0.4\dimexpr\macc@kerna
      \let\final@kern#2%
      \ifdim\dimen@<\z@ \let\final@kern1\fi
      \if\final@kern1 \kern-\dimen@\fi
    \else
      \overline{\rel@kern{-0.6}\kern\dimen@#1}%
    \fi
  }%
  \macc@depth\@ne
  \let\math@bgroup\@empty \let\math@egroup\macc@set@skewchar
  \mathsurround\z@ \frozen@everymath{\mathgroup\macc@group\relax}%
  \macc@set@skewchar\relax
  \let\mathaccentV\macc@nested@a
  \if#31
    \macc@nested@a\relax111{#1}%
  \else
    \def\gobble@till@marker##1\endmarker{}%
    \futurelet\first@char\gobble@till@marker#1\endmarker
    \ifcat\noexpand\first@char A\else
      \def\first@char{}%
    \fi
    \macc@nested@a\relax111{\first@char}%
  \fi
  \endgroup
}
\makeatother

\mapcitekey{Mukai:duality}{Mukai}
\mapcitekey{Pareschi+Popa:GV-sheaves}{PP}
\mapcitekey{Mumford:AbelianVarieties}{Mumford}
\mapcitekey{Hacon:GV}{Hacon}
\mapcitekey{Pareschi+Popa:regI}{PP-regI}
\mapcitekey{Pareschi+Popa:regIII}{PP-regIII}
\mapcitekey{Debarre:Ample}{Debarre}
\mapcitekey{Chen+Jiang:Positivity}{CJ}

\newcommand{\FM}{\operatorname{\mathsf{FM}}}
\newcommand{\Dbcoh}{D_{\mathit{coh}}^{\mathit{b}}}
\newcommand{\Dtcoh}[1]{D_{\mathit{coh}}^{#1}}
\renewcommand{\Ah}{A^{\vee}}
\newcommand{\OAh}{\shO_{\Ah}}

\newcommand{\fh}{\hat{f}}
\newcommand{\fhl}{\fh_{\ast}}
\newcommand{\fhu}{\fh^{\ast}}
\newcommand{\Bh}{B^{\vee}}
\newcommand{\Ph}{P}
\newcommand{\Lh}{\hat{L}}
\newcommand{\OB}{\shO_B}
\newcommand{\OBh}{\shO_{\Bh}}
\newcommand{\shFh}{\hat{\shF}}
\newcommand{\omXk}{\omega_{X/k}}
\newcommand{\op}{\mathit{op}}
\newcommand{\su}{s^{\ast}}
\DeclareMathOperator{\Coh}{Coh}

\newcommand{\OZ}{\shO_Z}
\DeclareMathOperator{\ev}{ev}

\begin{document}

\title{The Fourier-Mukai transform made easy}

\author[Ch.~Schnell]{Christian Schnell}
\address{Department of Mathematics, Stony Brook University, Stony Brook, NY 11794-3651}
\email{cschnell@math.sunysb.edu}

\begin{abstract}
We propose a slightly modified definition for the Fourier-Mukai transform (on abelian
varieties) that makes it much easier to remember various formulas. As an application, we
give relatively short proofs for two important theorems: the characterization of
GV-sheaves in terms of vanishing, due to Hacon; and fact that M-regularity implies
(continuous) global generation, due to Pareschi and Popa.
\end{abstract}
\date{\today}
\maketitle

\section{Introduction}

\newpar
The Fourier-Mukai transform, first introduced by Mukai \cite{Mukai}, is an
equivalence of $k$-linear triangulated categories
\[
	\derR \Phi_{P_A} \colon \Dbcoh(\OA) \to \Dbcoh(\OAh)
\]
between the bounded derived category of coherent sheaves on the abelian variety $A$,
and that on the dual abelian variety $\Ah = \Pic_{A/k}^0$. It is defined by pulling
back to $A \times \Ah$, tensoring by the Poincar\'e bundle $P_A$, and
then pushing forward to $\Ah$. 

\newpar
The purpose of this note is to propose a slightly different definition for the
Fourier-Mukai transform that makes it easier to remember various formulas. When $X$
is an equidimensional and nonsingular algebraic variety over $k$, we denote by 
\[
	\derR \Delta_X = \derR \shHom \bigl( \argbl, \omXk \decal{\dim X} \bigr)
		\colon \Dbcoh(\OX) \to \Dbcoh(\OX)^{\op}
\]
the (contravariant) Grothendieck duality functor. In the case $X = \Spec k$, we shall
use the simplified notation $\derR \Delta_k$.  

\begin{definition}
The exact functor
\[
	\FM_A = \derR \Phi_{P_A} \circ \derR \Delta_A \colon
		\Dbcoh(\OA) \to \Dbcoh(\OAh)^{\op}
\]
is called the \define{symmetric Fourier-Mukai transform}.
\end{definition}

The word ``symmetric'' will be justified by the results below. Note that $\FM_A$
is a \emph{contravariant} functor; this turns out to be quite useful in practice, for
instance in the study of GV-sheaves (see \parref{par:GV}).

\newpar
It is clear from Mukai's result \cite[Thm~2.2]{Mukai} that $\FM_A$ is also an
equivalence of categories. One
advantage of the new definition is that it respects the symmetry between
the two abelian varieties $A$ and $\Ah$. For example, one can show that
\begin{equation} \label{eq:example}
	\FM_A \bigl( k(0) \bigr) = \OAh \quad \text{and} \quad
	\FM_A(\OA) = k(0).
\end{equation}
Here $k(0) = e_{\ast} \shO_{\Spec k}$ means the structure sheaf of the closed point
$0 \in A(k)$; we use the same notation also on $\Ah$. 
The following theorem further justifies the name ``symmetric Fourier-Mukai transform''.

\begin{theorem} \label{thm:equivalence}
The composed functors $\FM_{\Ah} \circ \FM_A$ and $\FM_A \circ \FM_{\Ah}$
are naturally isomorphic to the identity. In other words,
\[
	\FM_A \colon \Dbcoh(\OA) \to \Dbcoh(\OAh)^{\op}
\]
is an equivalence of categories, with quasi-inverse $\FM_{\Ah}$.
\end{theorem}

This looks much simpler than Mukai's version \cite[Thm.~2.2]{Mukai}.

\newpar
The effect of pulling back or pushing forward by a homomorphism between
abelian varieties is easy to describe; for the case of isogenies, see
\cite[(3.4)]{Mukai}, and for the general case, \cite[Proposition~2.3]{CJ}.

\begin{proposition} \label{prop:homomorphism}
Let $f \colon A \to B$ be a homomorphism of abelian varieties over $k$. Then one has
natural isomorphisms of functors
\[
	\FM_B \circ \, \derR \fl = \derL \fhu \circ \FM_A 
		\quad \text{and} \quad
	\FM_A \circ \derL \fu = \derR \fhl \circ \FM_B,
\]
where $\fh \colon \Bh \to \Ah$ is the induced homomorphism between the dual
abelian varieties. 
\end{proposition}

\newpar
The symmetric Fourier-Mukai transform exchanges translations and tensoring by the
corresponding line bundles. Any closed point $a \in A(k)$ determines a translation
morphism $t_a \colon A \to A$; on closed points, it is given by the formula $t_a(x) =
a+x$. Compare the following result with \cite[(3.1)]{Mukai}.

\begin{proposition} \label{prop:translation}
Let $a \in A(k)$ and $\alpha \in \Ah(k)$ be closed points. Then one has natural
isomorphisms of functors
\[
	\FM_A \circ (t_a)_{\ast} = (\Ph_a \tensor \argbl) \circ \FM_A 
		\quad \text{and} \quad
	\FM_A \circ (P_{\alpha} \tensor \argbl) = (t_{\alpha})_{\ast} \circ \FM_A,
\]
where $P_a$ on $A$, and $P_{\alpha}$ on $\Ah$, are the corresponding line bundles.
\end{proposition}

Together with \eqref{eq:example}, this leads to the pleasant formulas
\begin{equation} \label{eq:example-2}
	\FM_A \bigl( k(a) \bigr) = P_a \quad \text{and} \quad
	\FM_A(P_{\alpha}) = k(\alpha),
\end{equation}
for any pair of closed points $a \in A(k)$ and $\alpha \in \Ah(k)$. 

\newpar \label{par:GV}
The symmetric Fourier-Mukai transform has the following simple effect on the standard
t-structure on $\Dbcoh(\OA)$. 

\begin{proposition} \label{prop:t-structure}
If $K \in \Dtcoh{\geq n}(\OA)$, then $\FM_A(K) \in \Dtcoh{\leq -n}(\OAh)$.
\end{proposition}

In particular, the symmetric Fourier-Mukai transform of a coherent sheaf is a
complex of coherent sheaves concentrated in non-positive degrees. This motivates the
following definition, first made by Pareschi and Popa \cite[Definition~3.1]{PP}.

\begin{definition} \label{def:GV}
A coherent sheaf $\shF \in \Coh(\OA)$ is called a \define{GV-sheaf} if 
\[
	\FM_A(\shF) = \shFh
\]
for a coherent sheaf $\shFh \in \Coh(\OAh)$. If $\shFh$ is torsion-free, then $\shF$
is called \define{M-regular}.
\end{definition}

Our definition of M-regularity differs from the original one in
\cite[Def.~2.1]{PP-regI}; the two definitions are equivalent by
\cite[Thm.~2.8]{PP-regIII}.

\newpar
\theoremref{thm:equivalence} shows that $\FM_{\Ah}(\shFh) = \shF$, and so the
coherent sheaf $\shFh$ is again a GV-sheaf on $\Ah$. The property of being a GV-sheaf
is therefore self-dual.  Using \propositionref{prop:homomorphism} and
\propositionref{prop:translation}, one shows that the set of closed points in the support
of the coherent sheaf $\shFh$ is equal to
\begin{equation} \label{eq:support}
	\bigl( \Supp \shFh \bigr)(k) 
		= \menge{\alpha \in \Ah(k)}%
		{H^0 \bigl( A, \shF \tensor P_{\alpha}^{-1} \bigr) \neq 0}.
\end{equation}
The minus sign is consistent with the formulas in \eqref{eq:example-2}.

\newpar
GV-sheaves have several other surprising properties. Let $\shF$ be a coherent sheaf
on $A$. For each $j \geq 0$, one has a reduced closed subscheme $S^j(A, \shF)
\subseteq \Ah$; its set of closed points is given by
\[
	S^j(A, \shF)(k) = \menge{\alpha \in \Ah(k)}
	{H^j(A, \shF \tensor P_{\alpha}^{-1}) \neq 0}.
\]
If $\shF$ is a GV-sheaf, then $S^0(A, \shF) = \Supp \shFh$, according to
\eqref{eq:support}. The following result is originally due to Hacon
\cite[Cor.~3.2]{Hacon}.

\begin{proposition} \label{prop:CSL}
If $\shF$ is a GV-sheaf on $A$, then
\[
	S^0(A, \shF) \supseteq S^1(A, \shF) \supseteq \dotsb \supseteq S^{\dim A}(A, \shF), 
\]
and moreover, one has $\codim S^j(A, \shF) \geq j$ for all $j \geq 0$.
\end{proposition}

\newpar
Another advantage of our definition is that it preserves positivity. Recall that an
ample line bundle $L$ on the abelian variety $A$ gives rise to a homomorphism
\begin{equation} \label{eq:phiL}
	\phi_L \colon A \to \Ah,
\end{equation}
with the property that $t_a^{\ast} L = L \tensor P_{\phi_L(a)}$ for every closed
point $a \in A(k)$. The following result is a version of \cite[Prop.~3.1]{Mukai}.

\begin{proposition} \label{prop:ample}
Let $L$ be an ample line bundle on $A$. Then $\Lh$ is an ample vector bundle of rank
$\dim H^0(A, L)$, and one has
\[
	\phi_L^{\ast} \Lh = \iu L \tensor H^0(A, L)^{\ast},
\]
where $i \colon A \to A$ is the inversion morphism.
\end{proposition}

Note that $\rk \Lh = \dim H^0(A, L)$ whereas $\rk L = \dim H^0(\Ah, \Lh)$; at least
for ample line bundles, the symmetric Fourier-Mukai transform therefore
interchanges ``rank'' and ``dimension of the space of global sections''. The
generalization to arbitrary GV-sheaves is that the generic rank of $\shFh$ is equal
to the Euler characteristic
\[
	\chi(A, \shF) = \sum_{j=0}^{\dim A} (-1)^j \dim H^j(A, \shF).
\]
This follows from \propositionref{prop:CSL}; compare also \cite[Cor.~2.8]{Mukai}.

\newpar
Finally, we note the following useful property of GV-sheaves, which follows almost
immediately from \propositionref{prop:t-structure}. As far as I know, this is a new
result.

\begin{proposition} \label{prop:GV}
Let $f \colon A \to B$ be a homomorphism of abelian varieties, and let $\shF$ be a
GV-sheaf on $A$. If $\fl \shF$ is again a GV-sheaf on $B$, then 
\[
	\FM_B(\fl \shF) = \fhu \shFh,
\]
where $\fh \colon \Bh \to \Ah$ is the induced homomorphism between the dual abelian
varieties.
\end{proposition}

\newpar
Although one can deduce all the results above from what is in Mukai's paper, we are 
going to prove everything from scratch. This will hopefully convince the reader that
the new definition is also easy to work with in practice.
As an application, we present relatively short proofs for two important
results about GV-sheaves and M-regular sheaves. The following theorem by Hacon
\cite[Thm~1.2]{Hacon} relates the GV-property to vanishing theorems for ample
line bundles.

\begin{theorem} \label{thm:Hacon}
A coherent sheaf $\shF \in \Coh(\OA)$ is a GV-sheaf if, and only if, for every
isogeny $\varphi \colon A' \to A$ and every ample line bundle $L'$ on $A'$, one has
\[
	H^i(A', L' \tensor \varphi^{\ast} \shF) = 0 \quad \text{for all $i > 0$.}
\]
\end{theorem}

\newpar
Recall that a coherent sheaf $\shF$ on an abelian variety is said to be
\define{continuously globally generated} \cite[Def~2.10]{PP-regI} if for every
nonempty Zariski-open subset $U \subseteq \Ah$, the evaluation morphism
\[
	\bigoplus_{\alpha \in U(k)} H^0 \bigl( A, \shF \tensor P_{\alpha}^{-1} \bigr)
		\tensor P_{\alpha} \to \shF
\]
is surjective. The first half of the following result is due to Pareschi and Popa
\cite[Prop~2.13]{PP-regI}, the second half to Debarre
\cite[Prop~3.1]{Debarre}.

\begin{theorem} \label{thm:M-regular}
Let $\shF \in \Coh(\OA)$ be an M-regular coherent sheaf on $A$.
\begin{aenumerate}
\item $\shF$ is continuously globally generated.
\item There is an isogeny $\varphi \colon A' \to A$ such that $\varphi^{\ast} \shF$
is globally generated.
\end{aenumerate}
\end{theorem}

In fact, one can prove more generally that if $\shF$ is a torsion-free coherent sheaf
on $A$, then $\shH^0 \FM_A(\shF)$ is continuously globally generated, and
globally generated after a suitable isogeny (see \theoremref{thm:torsion-free}
below). This is again a new result.

\newpar I thank Luigi Lombardi and Mihnea Popa for listening to my ideas about the
symmetric Fourier-Mukai transform, and for their comments about the first draft of
this note. While working out the details of the construction, I was partially
supported by grant DMS-1404947 from the National Science Foundation, and by a
Centennial Fellowship from the American Mathematical Society.

\section{Notation}

\newpar
Let $A$ be an abelian variety over an algebraically closed field $k$. We denote by 
\[
	m \colon A \times A \to A, \quad
	i \colon A \to A, \quad
	e \colon \Spec k \to A
\]
the $k$-morphisms representing composition, inversion, and the identity of the group
scheme $A$; we write $0 \in A(k)$ for the closed point corresponding to $e$.

We use the notation $\Ah = \Pic_{A/k}^0$ for the dual abelian variety. On the product
$A \times \Ah$, we have the Poincar\'e bundle $P_A$, normalized by the condition that
its pullback by $e \times \id$ is trivial; the Poincar\'e correspondence consists of
the line bundle $P_A$, together with fixed trivializations of $(e \times \id)^{\ast}
P_A$ and $(\id \times e)^{\ast} P_A$. 

Note that $\Pic_{\Ah/k}^0$ is canonically isomorphic to $A$, and that the Poincar\'e
correspondence on $\Ah \times A$ is the pullback of the Poincar\'e correspondence on
$A \times \Ah$, under the automorphism that swaps the two factors of the product. 

\newpar
For a closed point $\alpha \in \Ah(k)$, we denote by $P_{\alpha}$ the corresponding
line bundle on $A$; to be precise, $P_{\alpha}$ is the restriction of $P_A$ to the
closed subscheme $A \times \{\alpha\}$. Similary, $P_a$ means the line bundle on $\Ah$
corresponding to a closed point $a \in A(k)$. We also denote by
\[
	P_{(a, \alpha)} = P_A \restr{(a, \alpha)}
\]
the fiber of the Poincar\'e bundle at the closed point $(a, \alpha) \in (A \times
\Ah)(k)$. Using the canonical isomorphism $(i \times i)^{\ast} P_A = P_A$, we obtain
\begin{equation} \label{eq:P-a-alpha}
	P_{(a,\alpha)} = P_a \restr{\alpha} = P_{\alpha} \restr{a} =
	P_{(-a,-\alpha)} = P_{-\alpha} \restr{-a} = P_{-a} \restr{-\alpha}.
\end{equation}
All these isomorphisms are uniquely determined by the Poincar\'e correspondence.

\section{Proofs}

\newpar
We use the notation $\Dbcoh(\OA)$ for the derived category of cohomologically bounded
and coherent complexes of sheaves of $\OA$-modules; this category is equivalent to
the (much smaller) bounded derived category of $\Coh(\OA)$. In Mukai's original
definition, the Fourier-Mukai transform is the exact functor
\[
	\derR \Phi_{P_A} \colon \Dbcoh(\OA) \to \Dbcoh(\OAh),
\]
given by the (slightly abbreviated) formula $\derR \Phi_{P_A} = \derR (p_2)_{\ast} \bigl(
P_A \tensor \pu_1 \bigr)$. Here $p_1$ and $p_2$ are the projections from $A \times
\Ah$ to the two factors:
\[
\begin{tikzcd}
A \times \Ah \dar{p_1} \rar{p_2} & \Ah \\
A
\end{tikzcd}
\]

\newpar
Let us start by checking the two examples in \eqref{eq:example}. The first one is
very easy: Grothendieck duality, applied to the morphism $e \colon \Spec k \to A$, gives
\[
	\derR \Delta_A \bigl( k(0) \bigr) 
		= e_{\ast} \derR \Delta_{\Spec k} \bigl( \shO_{\Spec k} \bigr)
		= e_{\ast} \shO_{\Spec k} = k(0),
\]
and so the symmetric Fourier-Mukai transform of $k(0)$ is
\[
	\FM_A \bigl( k(0) \bigr) = \derR \Phi_{P_A} \bigl( k(0) \bigr) = \OA.
\]
The second isomorphism comes from the fixed trivialization of $(e
\times \id)^{\ast} P_A$. 

\newpar
The other example in \eqref{eq:example} needs more work. To help us describe
\begin{equation} \label{eq:complex}
	\FM_A(\OA) = \derR \Phi_{P_A} \Bigl( \omega_A \decal{\dim A} \Bigr)
		= \derR (p_2)_{\ast} \Bigl( P_A \tensor \pu_1 \omega_A \decal{\dim A} \Bigr),
\end{equation}
we first recall the following result about the cohomology of line bundles on abelian
varieties \cite[\S8(vii)]{Mumford}.

\begin{lemma} \label{lem:mumford}
If $\alpha \in \Ah(k)$ is nontrivial, then $H^n(A, P_{\alpha}) = 0$ for all $n \in
\NN$. 
\end{lemma}

\begin{proof}
The proof is by induction on $n \geq 0$. The case $n = 0$ is easy: if $s$ is a
nontrivial global section of $P_{\alpha}$, then $\iu s$ is a nontrivial global section of
$\iu P_{\alpha} = P_{-\alpha}$, and so $s \tensor \iu s$ is a nontrivial global
section of $P_{\alpha} \tensor P_{-\alpha} = \OA$; this clearly implies $\alpha
= 0$. To deal with the general case, we factor $\id \colon A \to A$ as
\[
\begin{tikzcd}
	A \rar{\id \times e} & A \times A \rar{m} & A.
\end{tikzcd}
\]
On the level of cohomology, we obtain a factorization of the identity as
\[
	H^n(A, P_{\alpha}) \to H^n \bigl( A \times A, m^{\ast} P_{\alpha} \bigr) 
		\to H^n(A, P_{\alpha}).
\]
Because $m^{\ast} P_{\alpha} = \pu_1 P_{\alpha} \tensor \pu_2 P_{\alpha}$, the
K\"unneth formula gives
\[
	H^n \bigl( A \times A, m^{\ast} P_{\alpha} \bigr) =
		\bigoplus_{i+j=n} H^i(A, P_{\alpha}) \tensor H^j(A, P_{\alpha}),
\]
which is zero by induction. The conclusion is that $H^n(A, P_{\alpha}) = 0$.
\end{proof}

\newpar
The remainder of the argument is taken from \cite[\S13]{Mumford}, with a few minor
tweaks. Consider the cohomology sheaves
\[
	\shF_j = R^j (p_2)_{\ast} \Bigl( P_A \tensor \pu_1 \omega_A \decal{\dim A} \Bigr)
		= R^{\dim A + j} (p_2)_{\ast} \bigl( P_A \tensor \pu_1 \omega_A \bigr)
\]
of the complex in \eqref{eq:complex}; they are coherent sheaves on the dual abelian
variety $\Ah$. Since the line bundle $P_A \tensor \pu_1 \omega_A$
is flat over $\Ah$, we can apply the base change theorem and \lemmaref{lem:mumford} to conclude that each
$\shF_j$ is supported on the closed point $0 \in \Ah(k)$; being coherent, $\shF_j$
therefore has finite length. For dimension reasons, we obviously have $\shF_j = 0$
for $j > 0$.

\newpar
The next task is to show that $\shF_j = 0$ also for $j < 0$. This can be done by
using Serre duality. Consider the Leray spectral sequence
\[
	E_2^{i,j} = H^i \bigl( \Ah, \shF_j \bigr) \Longrightarrow
		H^{\dim A + i+j} \bigl( A \times \Ah, P_A \tensor \pu_1 \omega_A \bigr).
\]
We have $E_2^{i,j} = 0$ for $i > 0$ (because $\Supp \shF_j$ is zero-dimensional); the
spectral sequence therefore degenerates at $E_2$ and gives us isomorphisms
\[
	H^{\dim A+j} \bigl( A \times \Ah, P_A \tensor \pu_1 \omega_A \bigr) 
		= H^0 \bigl( \Ah, \shF_j \bigr).
\]
In particular, this group vanishes for $j > 0$. By Serre duality, the $k$-vector
space on the left is dual to
\[
	H^{\dim A - j} \bigl( A \times \Ah, P_A^{-1} \tensor \pu_2 \omega_{\Ah} \bigr)
\]
and this vanishes for $j < 0$, for the same reason as before. But then $H^0
\bigl( \Ah, \shF_j \bigr) = 0$, and because the sheaf $\shF_j$ has finite length, we
see that $\shF_j = 0$ for every $j < 0$.

\newpar
So far, we have shown that the natural morphism
\[
	\derR (p_2)_{\ast} \Bigl( P_A \tensor \pu_1 \omega_A \decal{\dim A} \Bigr)
		\to \shF_0
\]
is an isomorphism. To conclude the proof of \eqref{eq:example}, we need to argue that
$\shF_0 = k(0)$. Another application of the base change theorem yields
\[
	e^{\ast} \shF_0 = H^{\dim A}(A, \omega_A) = k.
\]
By Nakayama's lemma, it follows that $\shF_0 = \OZ$ for a (maybe nonreduced) closed
subscheme $Z \subseteq \Ah$ supported on the closed point $0 \in \Ah(k)$. 
We now use the universal property of the Poincar\'e correspondence to show
that $Z$ is reduced, and hence that $\shF_0 = k(0)$. Let $\shF \in \Coh(\Ah)$ be an
arbitrary coherent sheaf.  Grothendieck duality, applied to the second projection
$p_2 \colon A \times \Ah \to \Ah$, gives us an isomorphism
\[
	\Hom \bigl( \OZ, \shF \bigr) =
	\Hom \bigl( \shF_0, \shF \bigr) = \Hom \bigl( P_A, \pu_2 \shF \bigr),
\]
functorial in $\shF$. In particular, the identity endomorphism of $\OZ$ corresponds
to a nontrivial morphism $P_A \to \pu_2 \OZ$, hence to a nontrivial
morphism 
\[
	P_A \restr{A \times Z} \to \shO_{A \times Z}. 
\]
Because the restriction of $P_A$ to the closed subscheme $A \times \{0\}$ is trivial,
Nakayama's lemma tells us that this morphism is an isomorphism. In other words, the
restriction of $P_A$ to the subscheme $A \times Z$ is trivial; by the universal
property of the Poincar\'e correspondence, the subscheme in question therefore has to
be contained in $A \times \{0\}$, which means exactly that $Z$ is reduced.

\newpar
Having shown that $\FM_A(\OA) = k(0)$, we can now explain why the symmetric
Fourier-Mukai transform is an equivalence of categories.

\begin{proof}[Proof of \theoremref{thm:equivalence}]
Since we can interchange the role of $A$ and $\Ah$, we only need to prove that the
functor
\[
	\FM_{\Ah} \circ \FM_A = \derR \Phi_{P_{\Ah}} \circ \, \derR \Delta_{\Ah} 
		\circ \, \derR \Phi_{P_A} \circ \, \derR \Delta_A
\]
is naturally isomorphic to the identity. To do this, we first consider
\[
	\derR \Delta_{\Ah} \circ \, \derR \Phi_{P_A} \circ \, \derR \Delta_A 
		\colon \Dbcoh(\OA) \to \Dbcoh(\OAh).
\]
A brief computation using Grothendieck duality shows that this functor is an integral
transform, whose kernel is the complex
\[
	P_A^{-1} \tensor \pu_2 \omega_{\Ah} \decal{\dim A}
\] 
on the product $A \times \Ah$. After composing with $\derR \Phi_{P_{\Ah}}$ and
swapping the order of the factors, we obtain another integral transform,
whose kernel is now the complex
\begin{equation} \label{eq:kernel}
	\derR (p_{12})_{\ast} \Bigl( \pu_3 \omega_{\Ah} \decal{\dim A} \tensor 
		\pu_{13} P_A^{-1} \tensor \pu_{23} P_A \Bigr)
\end{equation}
on $A \times A$. \theoremref{thm:equivalence} will be proved once we show that
\eqref{eq:kernel} is isomorphic to the structure sheaf of the diagonal in $A \times A$.

Let $s \colon A \times A \to A$ be defined as $s = m \circ (\id \times i)$; the
formula on closed points is $s(x,y) = x-y$. Considering $\pu_{13} P_A^{-1} \tensor
\pu_{23} P_A$ as a correspondence between $A \times A$ and $\Ah$, and using the
universal property of the Poincar\'e correspondence, we obtain
\[
	\pu_{13} P_A^{-1} \tensor \pu_{23} P_A = 
		(s \times \id)^{\ast} P_A.
\]
We can now apply flat base change in the commutative diagram
\[
\begin{tikzcd}
A \times A \times \Ah \dar{s \times \id} \rar{p_{12}} & A \times A \dar{s} \\
A \times \Ah \rar{p_1} & A
\end{tikzcd}
\]
and conclude that the complex in \eqref{eq:kernel} is isomorphic to
\[
	\su \, \derR (p_1)_{\ast} \Bigl( P_A \tensor \pu_2 \omega_{\Ah} \decal{\dim A} \Bigr)
	= \su \FM_{\Ah}(\OAh) = \su k(0) = \Delta_{\ast} \OA,
\]
where $\Delta \colon A \to A \times A$ is the diagonal embedding. Here we used the
fact that the symmetric Fourier-Mukai transform of the structure sheaf $\OAh$ is the
structure sheaf $k(0)$ of the closed point $0 \in A(k)$, by \eqref{eq:example}.
\end{proof}

\newpar
In this section, we investigate how the symmetric Fourier-Mukai transform interacts
with pushing forward and pulling back by homomorphisms between abelian varieties. Let $f
\colon A \to B$ be a homomorphism, and denote by $\fh \colon \Bh \to \Ah$ the induced
homomorphism between the dual abelian varieties.

\begin{proof}[Proof of \propositionref{prop:homomorphism}]
It will be enough to show that
\[
	\FM_B \circ \, \derR \fl = \derL \fhu \circ \FM_A;
\]
the second identity in the theorem follows from this with the help of
\theoremref{thm:equivalence}. 
Using the definition of $\FM_B$ and Grothendieck duality, we obtain
\[
	\FM_B \circ \, \derR \fl 
		= \derR \Phi_{P_A} \circ \, \derR \Delta_B \circ \, \derR \fl
		= \derR \Phi_{P_B} \circ \, \derR \fl \circ \, \derR \Delta_A.
\]
This reduces the problem to proving that
\begin{equation} \label{eq:relation}
	\derR \Phi_{P_B} \circ \, \derR \fl = \derL \fhu \circ \, \derR \Phi_{P_A}.
\end{equation}
We make use of the following commutative diagram:
\[
\begin{tikzcd}[row sep=large]
\Ah & \Bh \lar[swap]{\fh} \\
A \times \Ah \uar[swap,start anchor={[yshift=-0.2ex]},xshift=-0.2ex]{p_2} & A \times \Bh 
		\uar[swap,start anchor={[yshift=-0.5ex]},xshift=-0.5ex]{p_2} \lar[swap]{\id
\times \fh} \dar[xshift=-0.2ex]{p_1} \rar{f \times \id} 
		& B \times \Bh \dar[xshift=-0.2ex]{p_1} \rar{p_2} & \Bh \\
& A \rar{f} & B
\end{tikzcd}
\]
The universal property of the Poincar\'e correspondence gives $(f \times \id)^{\ast}
P_B = (\id \times \fh)^{\ast} P_A$. Using the projection formula and flat base change, we have
\begin{align*}
	\derR \Phi_{P_B} \circ \, \derR \fl &= 
	\derR (p_2)_{\ast} \bigl( P_B \tensor \pu_1 \derR \fl \bigr) 
	= \derR (p_2)_{\ast} \bigl( P_B \tensor \derR (f \times \id)_{\ast} \pu_1 \bigr) \\
	&= \derR (p_2)_{\ast} \bigl( (f \times \id)^{\ast} P_B \tensor \pu_1 \bigr)
	= \derR (p_2)_{\ast} \bigl( (\id \times \fh)^{\ast} P_A \tensor \pu_1 \bigr) \\
	&= \derR (p_2)_{\ast} \derL (\id \times \fh)^{\ast} \bigl( P_A \tensor \pu_1 \bigr)
	= \derL \fhu \, \derR (p_2)_{\ast} \bigl( P_A \tensor \pu_1 \bigr) \\
	&= \derL \fhu \circ \derR \Phi_{P_A}.
\end{align*}
This calculation establishes \propositionref{prop:homomorphism}.
\end{proof}

\newpar
In this section, we investigate the effect of translations by closed points.

\begin{proof}[Proof of \propositionref{prop:translation}]
Once again, it suffices to prove that
\[
	\FM_A \circ (t_a)_{\ast} = (\Ph_a \tensor \argbl) \circ \FM_A 
\]
because the other identity follows from this with the help of
\theoremref{thm:equivalence}. Using Grothendieck duality, we get a natural
isomorphism of functors
\[
	\FM_A \circ (t_a)_{\ast} 
		= \derR \Phi_{P_A} \circ \derR \Delta_A \circ (t_a)_{\ast}
		= \derR \Phi_{P_A} \circ (t_a)_{\ast} \circ \derR \Delta_A,
\]
and so the problem is reduced to showing that
\[
	\derR \Phi_{P_A} \circ (t_a)_{\ast} = (P_a \tensor \argbl) \circ \derR \Phi_{P_A}.
\]
We use the following commutative diagram:
\[
\begin{tikzcd}[row sep=large]
A \times \Ah \rar{t_a \times \id} \dar{p_1} & A \times \Ah \dar{p_1} \rar{p_2} & \Ah \\
A \rar{t_a} & A
\end{tikzcd}
\]
Since $(t_a \times \id)^{\ast} P_A = \pu_2 P_a \tensor P_A$, we have
\begin{align*}
	\derR \Phi_{P_A} \circ (t_a)_{\ast}
	&= \derR (p_2)_{\ast} \Bigl( P_A \tensor \pu_1 (t_a)_{\ast} \Bigr)
	= \derR (p_2)_{\ast} \Bigl( P_A \tensor (t_a \times \id )_{\ast} \pu_1 \Bigr) \\
	&= \derR (p_2)_{\ast} \Bigl( (t_a \times \id)^{\ast} P_A \tensor \pu_1 \Bigr)
	= \derR (p_2)_{\ast} \Bigl( \pu_2 P_a \tensor P_A \tensor \pu_1 \Bigr) \\
	&= P_a \tensor \derR (p_2)_{\ast} \bigl( P_A \tensor \pu_1 \bigr)
	= P_a \tensor \derR \Phi_{P_A},
\end{align*}
which is exactly what we needed.
\end{proof}

\newpar
In this section, we prove the assertions about GV-sheaves in the introduction.

\begin{proof}[Proof of \propositionref{prop:t-structure}] 
The functor $\FM_A$ is exact and contravariant, and so it suffices to show that
$\FM_A(\shF) \in \Dtcoh{\leq 0}(\OAh)$ for every coherent sheaf $\shF \in \Coh(\OA)$.
Observe that the $j$-th cohomology sheaf of the complex
\[
	\derR \Delta_A(\shF) = \derR \shHom \bigl( \shF, \omega_A \decal{\dim A} \bigr)
\]
is supported in codimension $\geq j + \dim A$; equivalently, the dimension of its support is
$\leq -j$. In the spectral sequence
\[
	E_2^{i,j} = R^i (p_2)_{\ast} \Bigl( P \tensor \pu_1 R^j \Delta_A(\shF) \Bigr)
		\Longrightarrow \shH^{i+j} \FM_A(\shF),
\]
we therefore have $E_2^{i,j} = 0$ whenever $i > -j$, for dimension reasons. We
conclude that $\shH^n \FM_A(\shF) = 0$ whenever $n > 0$, which is what we needed to
show.
\end{proof}

Suppose now that $\shF$ is a GV-sheaf; as in \definitionref{def:GV}, we let $\shFh =
\FM_A(\shF)$. Let us compute the restriction of the coherent sheaf $\shFh$ to a
closed point $\alpha \in \Ah(k)$. We can write this restriction in the form
\[
	\derL e^{\ast} (t_{-\alpha})_{\ast} \shFh,
\]
with $t_{-\alpha} \colon \Ah \to \Ah$ denoting translation by $-\alpha$. We can then
combine the identies in \propositionref{prop:homomorphism} and
\propositionref{prop:translation} to obtain an isomorphism
\begin{equation} \label{eq:restriction-derived}
	\derL e^{\ast} (t_{-\alpha})_{\ast} \FM_A(\shF)
		= \derL e^{\ast} \FM_A \bigl( P_{-\alpha} \tensor \shF \bigr)
		= \derR \Delta_k \derR \pl \bigl( P_{-\alpha} \tensor \shF \bigr),
\end{equation}
where $p \colon A \to \Spec k$ is the structure morphism. In particular
\begin{equation} \label{eq:restriction}
	\shFh \restr{\alpha} = \Hom_k \Bigl( H^0(A, \shF \tensor P_{-\alpha}), k \Bigr),
\end{equation}
and so $\alpha \in \Ah(k)$ belongs to the support of $\shFh$ if and only if $H^0(A,
\shF \tensor P_{-\alpha}) \neq 0$, proving \eqref{eq:support} from the introduction.

\newpar
This seems like a good place to prove \propositionref{prop:CSL} about the loci $S^j(A,
\shF)$.

\begin{proof}[Proof of \propositionref{prop:CSL}]
Let $\alpha \in \Ah(k)$ be a closed point. It will be convenient to use the symbol
$i_{\alpha} = t_{\alpha} \circ e \colon \Spec k \into \Ah$ for the resulting closed
embedding. After taking cohomology in \eqref{eq:restriction-derived}, we obtain
\[
	L^{-j} \iu_{\alpha} \shFh =
	\Hom_k \Bigl( H^j(A, \shF \tensor P_{-\alpha}), k \Bigr).
\]
This implies of course that
\[
	S^j(A, \shF)(k) = \menge{\alpha \in \Ah(k)}{L^{-j} \iu_{\alpha} \shFh \neq 0}.
\]
All the assertions in the proposition are now general facts about coherent sheaves on
nonsingular varieties. For the convenience of the reader, we recall the necessary
result (and its simple proof) in the following paragraph.
\end{proof}

\newpar
Let $X$ be a nonsingular algebraic variety over $k$; for a closed point $x \in X(k)$,
we denote by $i_x \colon \Spec k \into X$ the resulting closed embedding. For any
coherent sheaf $\shF$ on $X$, there are closed algebraic subsets $S^{-j}(\shF) \subseteq
X$, indexed by $j \in \NN$; on closed points, they are given by the formula
\[
	S^{-j}(\shF)(k) = \menge{x \in X}{L^{-j} \iu_x \shF \neq 0}.
\]
Here is the result that we used to prove \propositionref{prop:CSL}.

\begin{lemma}
Let $X$ be a nonsingular algebraic variety over $k$, and let $\shF$ be a coherent
sheaf on $X$. 
\begin{aenumerate}
\item One has $S^{-(j+1)}(\shF) \subseteq S^{-j}(\shF)$. 
\item Every irreducible component of $S^{-j}(\shF)$ has codimension $\geq j$. 
\end{aenumerate}
\end{lemma}

\begin{proof}
The basic fact is that every finitely generated module over a local ring $R$ has an
(essentially unique) minimal free resolution, and that when $R$ is regular, the
length of this resolution is at most $\dim R$. 

Let us prove the first assertion. It amounts to saying that $L^{-j} \iu_x \shF = 0$
implies $L^{-(j+1)} \iu_x \shF = 0$. The stalk $\shF_x$ is a finitely generated
module over the local ring $\shO_{X,x}$, and in its minimal free resolution, the rank
of the free $\shO_{X,x}$-module in degree $-j$ is equal to $\dim L^{-j} \iu_X \shF$.
If $L^{-j} \iu_X \shF = 0$, then the minimal free resolution has length at most $j$,
and this trivially implies that $L^{-(j+1)} \iu_x \shF = 0$.

Now let us prove the second assertion. Fix an irreducible component $Z$ of
$S^{-j}(\shF)$. After localizing at the generic point of $Z$, we obtain a finitely
generated module $\shF_Z$ over the local ring $\shO_{X, Z}$. By Serre's theorem,
$\shO_{X, Z}$ is a regular local ring of dimension $c = \codim(Z, X)$, and so the minimal
free resolution of $\shF_Z$ has lengt at most $c$. Since $\shF$ is coherent, there is
a Zariski-open subset $U \subseteq X$, containing the generic point of $Z$, such that
$\shF \restr{U}$ has a free resolution of length at most $c$. This implies that
$L^{-j} \iu_x \shF = 0$ for every $j > c$ and every $x \in U(k)$, and since $Z
\subseteq S^{-j}(\shF)$, the conclusion is that $c \geq j$.
\end{proof}

\newpar
We finish our discussion of the symmetric Fourier-Mukai transform and its
properties by proving \propositionref{prop:GV} from the introduction.

\begin{proof}[Proof of \propositionref{prop:GV}]
In $\Dbcoh(\OB)$, we have a distinguished triangle
\[
	\fl \shF \to \derR \fl \shF \to \tau_{\geq 1} \derR \fl \shF \to \dotsb,
\]
where $\tau_{\geq 1} \colon \Dbcoh(\OB) \to \Dtcoh{\geq 1}(\OB)$ is one of the
truncation functors for the standard t-structure. Applying $\FM_B$, which is exact
and contravariant, we obtain another distinguished triangle
\[
	\FM_B \bigl( \tau_{\geq 1} \derR \fl \shF \bigr)
		\to \FM_B(\derR \fl \shF) \to \FM_B(\fl \shF) \to \dotsb
\]	
in $\Dbcoh(\OBh)$. The complex on the left belongs to $\Dtcoh{\leq -1}(\OBh)$, due to
\propositionref{prop:t-structure}; also $\FM_B(\derR \fl \shF) = \derL \fhu
\FM_A(\shF) = \derL \fhu \shFh$ by \propositionref{prop:homomorphism}. Going from the
distinguished triangle to the long exact sequence in cohomology leads to
\[
	\fhu \shFh = \shH^0 \FM_B(\fl \shF) = \FM_B(\fl \shF),
\]
due to the fact that $\fl \shF$ is also a GV-sheaf.
\end{proof}	

\section{Applications} 

\newpar
Now let us turn to the applications to GV-sheaves and M-regular sheaves. We start by
proving two simple lemmas. Recall the notation 
\[
	K_1 \boxtimes K_2 = \pu_1 K_1 \tensor \pu_2 K_2 \in \Dbcoh(\shO_{A \times B})
\]
for objects $K_1 \in \Dbcoh(\OA)$ and $K_2 \in \Dbcoh(\OB)$.

\begin{lemma} \label{lem:box}
We have a natural isomorphism of bifunctors
\[
	\FM_{A \times B} \circ \, \boxtimes = \FM_A \boxtimes \FM_B
\]
from $\Dbcoh(\OA) \times \Dbcoh(\OB)$ to $\Dbcoh(\OAh) \times \Dbcoh(\OBh)$.
\end{lemma}

\begin{proof}
	This follows from the isomorphism $P_{A \times B} = \pu_{13} P_A
	\tensor \pu_{24} P_B$ on the abelian variety $A \times B \times \Ah \times \Bh$,
	by the usual base change arguments.
\end{proof}

\newpar
The following lemma shows the effect of tensoring by an ample line bundle.

\begin{lemma} \label{lem:tensor-ample}
Let $L$ be an ample line bundle on $A$, and $K \in \Dbcoh(\OA)$. Then the
$k$-vector space $\Hom_k \bigl( H^i(A, L \tensor K), k \bigr)$
is isomorphic to a direct summand of
\[
	H^{-i} \bigl( \Ah, L \tensor \phi_L^{\ast} \FM_A(K) \bigr) \tensor H^0(A, L)^{\ast},
\]
where $\phi_L \colon A \to \Ah$ is the homomorphism in \eqref{eq:phiL}.
\end{lemma}

\begin{proof}
If we apply \propositionref{prop:homomorphism} to the morphism $p \colon A \to \Spec
k$, we obtain
\[
	\derR \Delta_{\Spec k} \derR \pl(L \tensor K) =
	\derL e^{\ast} \FM_A(L \tensor K).
\]
Now $L \tensor K = \derL \Delta^{\ast}(L \boxtimes K)$, and the diagonal
homomorphism $\Delta \colon A \to A \times A$ is dual to the multiplication morphism
$m \colon \Ah \times \Ah \to \Ah$. Another application of
\propositionref{prop:homomorphism}, followed by \lemmaref{lem:box}, therefore gives
\[
	\FM_A(L \tensor K) = \FM_A \derL \Delta^{\ast}(L \boxtimes K)
	= \derR m_{\ast} \bigl( \Lh \boxtimes \FM_A(K) \bigr),
\]
due to the fact that $L$ is a GV-sheaf. In the rest of the proof, we use the
following commutative diagram:
\[
\begin{tikzcd}
	A \rar{\phi_L} \drar[bend right=20]{p} & \Ah \dar{p} \rar{(i, \id)} & 
	\Ah \times \Ah \dar{m} \\
	 & \Spec k \rar{e} & \Ah
\end{tikzcd}
\]
Base change for the smooth morphism $m$ yields an isomorphism
\[
	\derR \Delta_k \derR \pl(L \tensor K) =
	\derR \pl \bigl( \iu \Lh \tensor \FM_A(K) \bigr).
\]
Since $\phi_L$ is an isogeny, $\derR \pl \bigl( \iu \Lh \tensor \FM_A(K) \bigr)$ is
isomorphic to a direct summand of
\[
	\derR (p \circ \phi_L)_{\ast} \phi_L^{\ast} \bigl( \iu \Lh \tensor \FM_A(K) \bigr).
\]
Using \propositionref{prop:ample}, we can rewrite this in the form
\[
	\derR \pl \bigl( \phi_L^{\ast} \iu \Lh \tensor \phi_L^{\ast} \FM_A(K) \bigr)
	= \derR \pl \bigl( L \tensor \phi_L^{\ast} \FM_A(K) \bigr) \tensor H^0(A,
	L)^{\ast}.
\]
We now get the desired result by taking cohomology.
\end{proof}

\newpar 
We can now prove Hacon's criterion for GV-sheaves in terms of vanishing.

\begin{proof}[Proof of Theorem~\ref*{thm:Hacon}]
Suppose that $\shF$ is a GV-sheaf on $A$. The pullback of $\shF$ by an isogeny
$\varphi \colon A' \to A$ is again a GV-sheaf on $A'$; this follows from the
identities in \propositionref{prop:homomorphism} because
\[
	\FM_{A'}(\varphi^{\ast} \shF) = \hat{\varphi}_{\ast} \FM_A(\shF).
\]
Here we used the fact that isogenies are flat and finite. We can therefore
assume without loss of generality that $\varphi = \id_A$. Now let $L$ be any ample
line bundle on $A$. To prove the first half of the theorem, we need to show that
\[
	H^i(A, L \tensor \shF) = 0 \quad \text{for $i > 0$.}
\]
\lemmaref{lem:tensor-ample} reduces the problem to showing that 
\[
	H^{-i} \bigl( \Ah, L \tensor \phi_L^{\ast} \FM_A(\shF) \bigr) = 0
	\quad \text{for $i > 0$.}
\]
But this is obvious because $\FM_A(\shF)$ is a sheaf (and $\phi_L$ is flat).

Next, we prove the converse. Suppose that $\shF \in \Coh(\OA)$ has the property
stated in the theorem. We need to explain why $\shH^i \FM_A(\shF) = 0$ for every $i
\neq 0$; in fact, by \propositionref{prop:t-structure}, it would be enough to
consider only $i < 0$. After choosing a sufficiently ample line bundle $L'$ on $\Ah$, the
problem reduces to showing that
\[
	H^i \bigl( \Ah, L' \tensor \FM_A(\shF) \bigr) = 0
	\quad \text{for $i \neq 0$.}
\]
Another application of \lemmaref{lem:tensor-ample}, but with the roles of $A$ and
$\Ah$ interchanged, shows that it is enough to prove the vanishing
\[
	H^{-i} \bigl( \Ah, L' \tensor \phi_{L'}^{\ast} \shF \bigr) = 0
	\quad \text{for $i \neq 0$,}
\]
where $\phi_{L'} \colon \Ah \to A$ is the isogeny determined by $L'$. This is obvious
for $i > 0$, and for $i < 0$, it is exactly the condition in the statement of the
theorem. 
\end{proof}

\newpar
In the remainder of the paper, we relate M-regularity to (continuous) global generation. The
underlying idea is extremely simple. Suppose that $\shF \in \Coh(\OA)$ is a GV-sheaf,
and set $\shFh = \FM_A(\shF) \in \Coh(\OAh)$. Let $S \subseteq \Ah(k)$ be any set of
closed points. A small calculation shows that the evaluation morphism
\begin{equation} \label{eq:S-evaluation}
	\bigoplus_{\alpha \in S} 
	H^0(A, \shF \tensor P_{-\alpha}) \tensor P_{\alpha} \to \shF
\end{equation}
corresponds, under the (contravariant!) functor $\FM_A$, to the morphism
\begin{equation} \label{eq:S-dual-evaluation}
	\shFh \to \prod_{\alpha \in S} \shFh \tensor k(\alpha);
\end{equation}
here $k(\alpha) = (t_{\alpha} \circ e)_{\ast} \shO_{\Spec k}$. Now suppose that
$\shF$ is M-regular, and that the set $S \subseteq \Ah(k)$ is dense in the Zariski topology.
Since $\shFh$ is torsion-free, the morphism in \eqref{eq:S-dual-evaluation} must be
injective, and so we should expect the evaluation morphism in \eqref{eq:S-evaluation}
to be surjective. The only problem is that infinite sums are not coherent (and
infinite products are not even quasi-coherent), which means that our results about
$\FM_A$ do not apply. This can be dealt with by making a pointwise argument.

\newpar
We now turn the above idea into a rigorous proof. 

\begin{proposition} \label{prop:evaluation}
Suppose that $\shF \in \Coh(\OA)$ is M-regular. Let $S \subseteq \Ah(k)$ be any
subset that is dense in the Zariski topology. Then there is a finite subset $F
\subseteq S$ such that the evaluation morphism
\[
	\bigoplus_{\alpha \in F}
	H^0(A, \shF \tensor P_{-\alpha} \bigr) \tensor P_{\alpha} \to \shF
\]
is surjective.
\end{proposition}

\begin{proof}
Let $a \in A(k)$ be any closed point. Consider the evaluation morphism
\[
	\ev_{\alpha} \colon H^0(A, \shF \tensor P_{-\alpha}) \tensor P_{\alpha} \to \shF.
\]
To show that $\ev_{\alpha}$ is surjective in a neighborhood of a closed point $a \in
A(k)$, it is enough (by Nakayama's lemma) to show that the induced morphism
\begin{equation} \label{eq:evaluation-A}
	H^0(A, \shF \tensor P_{-\alpha}) \to (\shF \tensor P_{-\alpha}) \restr{a}
\end{equation}
is surjective. By \propositionref{prop:translation}, we have
\[
	\FM_A(\shF \tensor P_{-\alpha}) = (t_{-\alpha})_{\ast} \FM_A(\shF)
		= (t_{-\alpha})_{\ast} \shFh.
\]
Combining this with the formula in \eqref{eq:restriction}, we get
\begin{align*}
	(\shF \tensor P_{-\alpha}) \restr{a} &=
	\Hom_k \bigl( H^0(\Ah, (t_{-\alpha})_{\ast} \shFh \tensor P_{-a}), k \bigr) \\
	&= \Hom_k \bigl( H^0(\Ah, \shFh \tensor P_{-a}), k \bigr),
\end{align*}
using the translation invariance of $P_{-a}$. It follows that the morphism 
in \eqref{eq:evaluation-A} is dual to the morphism
\begin{equation} \label{eq:evaluation-Ah}
	H^0(\Ah, \shFh \tensor P_{-a}) \to (\shFh \tensor P_{-a}) \restr{\alpha}.
\end{equation}
The conclusion is that $\ev_{\alpha}$ is surjective in a neighborhood of a closed
point $a \in A(k)$ if and only if the morphism in \eqref{eq:evaluation-Ah} is
injective.

Now $H^0(\Ah, \shFh \tensor P_{-a})$ is a finite-dimensional $k$-vector space. Since
$\shFh$ is a torsion-free coherent sheaf, and $S \subseteq \Ah(k)$ is dense in the
Zariski topology, we can therefore find a \emph{finite} subset $S(a)
\subseteq S$ such that the morphism
\[
	H^0(\Ah, \shFh \tensor P_{-a}) \to 
	\prod_{\alpha \in S(a)} (\shFh \tensor P_{-a}) \restr{\alpha}
\]
is injective. It follows that the dual morphism
\[
	\bigoplus_{\alpha \in S(a)}
		H^0(A, \shF \tensor P_{-\alpha}) \to (\shF \tensor P_{-\alpha}) \restr{a}
\]
is surjective, and hence that the evaluation morphism
\[
	\bigoplus_{\alpha \in S(a)}
	H^0(A, \shF \tensor P_{-\alpha}) \tensor P_{\alpha} \to \shF
\]
is surjective on a Zariski-open subset containing the given closed point $a \in
A(k)$. Finitely many of these open subsets cover $A$;
consequently, there is a finite subset $F \subseteq S$ with the property that
\[
	\bigoplus_{\alpha \in F}
	H^0(A, \shF \tensor P_{-\alpha}) \tensor P_{\alpha} \to \shF
\]
is surjective. This completes the proof.
\end{proof}

\newpar
We can now show that M-regularity implies (continuous) global generation.

\begin{proof}[Proof of \theoremref{thm:M-regular}]
The first assertion follows immediately from \propositionref{prop:evaluation}. For
the second one, let $S \subseteq \Ah(k)$ be the set of all torsion points; this is
certainly dense in the Zariski topology. By \propositionref{prop:evaluation}, there
is a finite subset $F \subseteq S$ such that
\[
	\bigoplus_{\alpha \in F}
	H^0(A, \shF \tensor P_{-\alpha}) \tensor P_{\alpha} \to \shF
\]
is surjective. Let $\varphi \colon A' \to A$ be an isogeny with the property that
$\hat{\varphi}(\alpha) = 0$ for every $\alpha \in F$; for example, multiplication by
the greatest common divisor of the orders of the points in $F$ will do. The induced
morphism
\[
	\bigoplus_{\alpha \in F}
	H^0(A, \shF \tensor P_{-\alpha}) \tensor \shO_{A'} \to \varphi^{\ast} \shF
\]
is then still surjective, and so $\varphi^{\ast} \shF$ is globally generated.
\end{proof}

\newpar
In fact, the argument above leads to the following stronger result.

\begin{theorem} \label{thm:torsion-free}
Suppose that $\shF \in \Coh(\OA)$ is torsion-free. Let $S \subseteq A(k)$ be any
subset that is dense in the Zariski topology. Then there is a finite subset $F
\subseteq S$ such that the evaluation morphism
\[
	\bigoplus_{a \in F}
	H^0 \bigl( \Ah, \shH^0 \FM_A(\shF) \tensor P_{-a} \bigr) \tensor P_a
	\to \shH^0 \FM_A(\shF)
\]
is surjective.
\end{theorem}

\begin{proof}
Let $a \in A(k)$ and $\alpha \in \Ah(k)$ be any pair of closed points. We have
\begin{align*}
	\derR \Delta_k \derL e^{\ast} (t_{-a})_{\ast}(\shF \tensor P_{-\alpha}) 
	&= \derR \pl \bigl( (t_{-\alpha})_{\ast} \FM_A(\shF) \tensor P_{-a} \bigr) \\
	&= \derR \pl \bigl( \FM_A(\shF) \tensor t_{-\alpha}^{\ast} P_{-a} \bigr)
\end{align*}
using \propositionref{prop:homomorphism}, \propositionref{prop:translation}, and the
projection formula (for $t_{-\alpha}$). Our fixed trivializations for $(e
\times \id)^{\ast} P_A$ and $(\id \times e)^{\ast} P_A$ determine an isomorphism
\[
	(\id \times t_{-\alpha})^{\ast} P_A = \pu_1 P_{-\alpha} \tensor P_A,
\]
which we can use to rewrite the identity above in the form
\[
	\derR \Delta_k \derL e^{\ast} (t_{-a})_{\ast}(\shF \tensor P_{-\alpha}) 
	= \derR \pl \bigl( \FM_A(\shF) \tensor P_{-a} \bigr) \tensor P_{(-a,-\alpha)}.
\]
Here $P_{(-a,-\alpha)}$ is the fiber of the Poincar\'e bundle at $(-a, -\alpha) \in
(A \times \Ah)(k)$. Taking cohomology in degree zero, we see that
\[
	(\shF \tensor P_{-\alpha}) \restr{a} \quad \text{is dual to} \quad
	H^0 \bigl( \Ah, \FM_A(\shF) \tensor P_{-a} \bigr) \tensor P_a \restr{\alpha},
\]
using the identity in \eqref{eq:P-a-alpha} to replace $P_{(-a,-\alpha)}$. Similarly,
we have
\[
	\derR \Delta_k \derR \pl \bigl( P_{-\alpha} \tensor \shF \bigr)
	= \derL e^{\ast} (t_{-\alpha})_{\ast} \FM_A(\shF),
\]
and since $\FM_A(\shF) \in \Dtcoh{\leq 0}(\OAh)$ by
\propositionref{prop:t-structure}, it follows that
\[
	H^0(A, \shF \tensor P_{-\alpha}) \quad \text{is dual to} \quad
	L^0 e^{\ast} (t_{-\alpha})_{\ast} \FM_A(\shF) = \shH^0 \FM_A(\shF) \restr{\alpha}.
\]
The conclusion is that that the evaluation morphism
\[
	H^0(A, \shF \tensor P_{-\alpha}) \to (\shF \tensor P_{-\alpha}) \restr{a}
\]
on $A$ is dual to the evaluation morphism
\[
	H^0 \bigl( \Ah, \FM_A(\shF) \tensor P_{-a} \bigr) \tensor P_a \restr{\alpha}
	\to L^0 e^{\ast} (t_{-\alpha})_{\ast} \FM_A(\shF).
\]
on $\Ah$. Note that the second morphism fits into a commutative diagram
\[
\begin{tikzcd}
	H^0 \bigl( \Ah, \FM_A(\shF) \tensor P_{-a} \bigr) \tensor P_a \restr{\alpha}
	\rar \dar & L^0 e^{\ast} (t_{-\alpha})_{\ast} \FM_A(\shF) \dar[equal] \\
	H^0 \bigl( \Ah, \shH^0 \FM_A(\shF) \tensor P_{-a} \bigr) \tensor P_a \restr{\alpha}
	\rar  & \shH^0 \FM_A(\shF) \restr{\alpha}.
\end{tikzcd}
\]
Now we can argue as before to prove the assertion. 

Fix a closed point $\alpha \in \Ah(k)$. Since $\shF$ is torsion-free, and $S
\subseteq A(k)$ is dense in the Zariski topology, there is a finite subset $S(\alpha)
\subseteq S$ such that the morphism
\[
	H^0(A, \shF \tensor P_{-\alpha}) \to \prod_{a \in S(\alpha)} (\shF \tensor
	P_{-\alpha}) \restr{a}
\]
is injective. After dualizing, it follows that the morphism
\[
	\bigoplus_{a \in S(\alpha)} H^0 \bigl( \Ah, \shH^0 \FM_A \tensor P_{-a} \bigr)
	\tensor P_a \restr{\alpha} \to \shH^0 \FM_A(\shF) \restr{\alpha}
\]
is surjective. By Nakayama's lemma, this means that
\[
	\bigoplus_{a \in S(\alpha)} H^0 \bigl( \Ah, \shH^0 \FM_A \tensor P_{-a} \bigr)
	\tensor P_a \to \shH^0 \FM_A(\shF)
\]
is surjective on a Zariski-open subset containing the closed point $\alpha \in
\Ah(k)$. Finitely many of these open subsets cover $\Ah$;
consequently, there is a finite subset $F \subseteq S$ with the property that
\[
	\bigoplus_{a \in F} H^0 \bigl( \Ah, \shH^0 \FM_A \tensor P_{-a} \bigr)
	\tensor P_a \to \shH^0 \FM_A(\shF)
\]
is surjective. This finishes the proof.
\end{proof}

\providecommand{\bysame}{\leavevmode\hbox to3em{\hrulefill}\thinspace}
\providecommand{\MR}{\relax\ifhmode\unskip\space\fi MR }
\providecommand{\MRhref}[2]{%
  \href{http://www.ams.org/mathscinet-getitem?mr=#1}{#2}
}
\providecommand{\href}[2]{#2}


\end{document}